\documentclass{article} %

\usepackage[latin1]{inputenc}
\usepackage{subfigure}

\author{Robert \v S\'amal
  \thanks{Email: \texttt{samal@iuuk.mff.cuni.cz.}    
    Partially supported by grant GA \v{C}R P202-12-G061. 
    Partially supported by grant LL1201 ERC CZ of the Czech Ministry of Education, Youth and Sports.}
  \thanks{ Computer Science Institute, Charles University in Prague }
}
\title{Cubical coloring --- fractional covering by cuts and semidefinite programming
\thanks{Preliminary version of this research appeared as extended abstract~\cite{chiq-ea}.}
}
\date{}

\usepackage {latexsym} \usepackage {amsmath} \usepackage {amsfonts} \usepackage {amssymb} \usepackage {times} \usepackage {longtable}

\def\zet{{\mathbb Z}}

\def\Q{{\mathbb Q}}
\def\R{{\mathbb R}}
\def\MC{\mathop{\rm MAXCUT}\nolimits}
\let\eps\varepsilon

\def\scal   <#1,#2>{\langle #1,#2 \rangle}
\def\numTT  (#1,#2){\langle #1,#2 \rangle_{TT}}
\def\numTTi (#1,#2){\langle\!\langle #1,#2 \rangle_{TT}}
\def\numTTs (#1,#2){\langle #1,#2 \rangle\!\rangle_{TT}}
\def\numFF  (#1,#2){\langle #1,#2 \rangle_{FF}}
\def\numFFi (#1,#2){\langle\!\langle #1,#2 \rangle_{FF}}
\def\numFFs (#1,#2){\langle #1,#2 \rangle\!\rangle_{FF}}
\def\numXX  (#1,#2){\langle #1,#2 \rangle_{XX}}
\def\numXXi (#1,#2){\langle\!\langle #1,#2 \rangle_{XX}}
\def\numXXs (#1,#2){\langle #1,#2 \rangle\!\rangle_{XX}}

\let\phi\varphi
\def\?#1{\relax}

\def\zet{{\mathbb Z}}

\def\Pet{\mathord{\mathrm {Pt}}}

\let\eps\varepsilon
\def\Exp{\mathbb{E}}

\DeclareMathAccent{\myarrow}{\mathord}{letters}{"7E}

\def\chiq{\chi_q}

\newcommand{\generalXY}[2][]{\mathrel{\xrightarrow{\ifx @#1@ #2\else #2_{#1} \fi}}}
\newcommand{\ngeneralXY}[2][]{\mathrel{\thickspace\not\negthickspace\xrightarrow{\ifx @#1@ #2\else #2_{#1} \fi}}}

\renewcommand{\hom}{\generalXY{hom}}

\def\ceil #1{\left\lceil #1 \right\rceil}

\newtheorem{theorem}{Theorem}[section]
\newtheorem{lemma}[theorem]{Lemma}

\newtheorem{observation}[theorem]{Observation}

\newtheorem{corollary}[theorem]{Corollary}

\newtheorem{question}[theorem]{Question}
\newtheorem{conjecture}[theorem]{Conjecture}
\newenvironment{proof}{\par\medskip\noindent{\bf Proof: }}
  {\hskip 2cm\unskip\hbox{}\hfill$\Box$\par\bigskip}

\begin{document}
\maketitle
\begin{abstract}
We introduce a new graph parameter that measures fractional covering of a graph by cuts. 
Besides being interesting in its own right, it is useful for study of 
homomorphisms and tension-continuous mappings.
We study the relations with chromatic number, bipartite density, 
and other graph parameters. 

We find the value of our parameter for a family of 
graphs based on hypercubes. These graphs play for our parameter 
the role that cliques play for the chromatic number 
and Kneser graphs for the fractional chromatic number. 
The fact that the defined parameter attains on these graphs 
the correct value suggests that our definition is a natural one. 
In the proof we use the eigenvalue bound for maximum cut
and a recent result of Engström, Färnqvist, Jonsson, and Thapper
[An approximability-related parameter on graphs -- properties and applications, DMTCS vol. 17:1, 2015, 33--66].  

We also provide a polynomial time approximation algorithm based 
on semidefinite programming and in particular on vector chromatic 
number (defined by Karger, Motwani and Sudan 
[Approximate graph coloring by semidefinite programming, J. ACM 45 (1998), no.~2, 246--265]).  
\end{abstract}


\section{Introduction}

All graphs we consider are undirected and loopless; to avoid trivialities
we do not consider edgeless graphs. 
For a set $W \subseteq V(G)$ we let $\delta(W)$ denote the
set of edges leaving~$W$ and we call any set of form $\delta(W)$ a
\emph{cut}.
Other terminology we shall be using is standard, and can be found 
in, e.g., \cite{Diestel}.

Let us call a (cut) \emph{$n/k$-cover of~$G$} an $n$-tuple 
$(X_1, \dots, X_n)$ of cuts in~$G$ such that every edge of~$G$ is covered by at
least~$k$ of them. We define two closely related parameters of $G$. We let 
$$
  x(G) = \inf \Bigl\{ \frac nk \mid \hbox {exists $n/k$-cover of~$G$} \Bigr\}
$$
and call $x(G)$ the \emph{fractional cut-covering number} of~$G$.
Its rescaling 
$$
  \chiq(G) = \frac{2}{2-x(G)} 
$$ 
will be called the \emph{cubical chromatic number} of~$G$. 
This terminology is motivated by analogy with the fractional and circular 
chromatic number, see the discussion following Equation~\eqref{eq:chiqalt} below. 
The rescaling function $2/(2-x)$  
serves the purpose of aligning the value with other variants 
of chromatic number, namely of attaining the right value for complete graphs. 
However, the rescaling function is far from arbitrary, as 
the values for other graphs are also modified in a proper way, 
see Theorem~\ref{t:chivchiq}. 

If $k=1$, i.e., if we want to cover every edge at least once,
then we need at least $\ceil{\log_2 \chi(G)}$ of them~\cite{DNR}. 
Here we consider a fractional version. In this
context we may find it surprising that $x(G) < 2$ for every~$G$ 
(Corollary~\ref{xbounds}). 

From another perspective, $x(G)$ is the fractional chromatic number 
of a certain hypergraph: it has $E(G)$ as points and odd cycles
of $G$ as hyperedges. 
This suggests that $x(G)$ is a solution of a linear program, 
see Equations~\eqref{eq:LP1} and~\eqref{eq:LP2} below.

The parameter $x(G)$ has found surprising use in theoretical computer
science. Färnqvist, Jonsson, and Thapper~\cite{Svedi1} 
study the approximability of MAXCUT and its generalizations (so-called MAX-$H$-COLORING)
using a suitably defined metric space. The function used to define the
metric is in \cite{Svedi2} recognized as a natural generalization of 
fractional covering by cuts. See the concluding remarks for further discussion. 

As another point of view we note that $x(G)$ is 
a certain type of chromatic number, but instead of complete graphs
(or Kneser graphs or circulants) which are used to define chromatic number 
(or fractional or circular chromatic number) it uses another graph scale. 
Let $Q_{n/k}$ denote a graph with $\{0,1\}^n$ as the set of vertices, 
where $xy$ forms an edge iff $d(x,y) \ge k$ (here $d(x,y)$ is the
Hamming distance of $x$~and $y$). 

\begin{observation}
A graph has $n/k$-cover
if and only if it is homomorphic to $Q_{n/k}$. 
\end{observation}

\begin{proof}
If $(X_1, \dots, X_n)$ is a cut $n/k$-cover of a graph $G$ then 
we can define homomorphism $f:V(G) \to V(Q_{n/k})$ as follows: 
for each $i$ we write $X_i$ as $\delta(W_i)$; we put 
$f(v) = 1$ if $v \in W_i$ and $f(v)=0$ otherwise. Now 
$f=(f_1, \dots, f_n)$ is a homomorphism. If, on the other hand, 
we are given a homomorphism $f:V(G) \to V(Q_{n/k})$ then we 
can put $W_i = \{ v\in V(G): f_i(v)=1\}$ and observe that 
$(\delta(W_1), \dots, \delta(W_n))$ is a cut $n/k$-cover. 
\end{proof}

The above observation implies that an alternative definition of $x(G)$ is
\begin{equation} \label{eq:chiqalt}
  x(G) = \inf \Bigl\{ \frac nk \mid  G \hom Q_{n/k}  \Bigr\} \,.
\end{equation}
This is analogous to the definition of fractional chromatic number by means of 
homomorphisms to Kneser graphs (or of circular chromatic number by circulants). 
An immediate corollary is that $x(G)$ is a homomorphism invariant, 
that is if $G \hom H$ then $x(G) \le x(H)$. This will be strengthened
in Lemma~\ref{xTT}. This observation suggests a possible use of the cubical chromatic number 
to study the structure of graph homomorphisms---we can prove nonexistence of 
a homomorphism $G \hom H$ by showing that $x(G) > x(H)$. 

For a graph $H$ let $H^{{}^\ge k}$ denote the graph with vertices $V(H)$ and edges $uv$ 
for any $u,v\in V(H)$ with distance in~$H$ at least~$k$. Further let
$Q_n$ denote the $n$-dimensional cube. Then $Q_{n/k}=Q_n^{{}^\ge k}$. 
This corresponds to the definition of circular chromatic number, where
the target graph is~$C_n^{{}^\ge k}$. 
This observation inspires the term cubical chromatic number. However, as we will
see below (in Corollary~\ref{xbounds}), a rescaling of $x(G)$ is in order 
to make it behave like a version of chromatic number, thus the definition
of~$\chiq(G)$.

The original motivation for defining $x(G)$ was the study 
\cite{rs-thesis,NS-TT} of cut-continuous mappings (defined in \cite{DNR}).
Given graphs $G$, $H$ we call a mapping
$f: E(G) \to E(H)$ \emph{cut-continuous}, if for every cut
$U \subseteq E(H)$, the preimage $f^{-1}(U)$ is a cut in~$G$.
The following lemma is straightforward, but useful.

\begin {lemma}   \label{xTT}
Let $G$, $H$ be graphs. Then if there is a cut-continuous
mapping from $G$ to $H$ (in particular, if there
is a homomorphism $G \hom H$), then $x(G) \le x(H)$
and (equivalently) $\chiq(G) \le \chiq(H)$.
\end {lemma}

\begin {proof}
It suffices to show that whenever $H$ has an $n/k$-cover, $G$ has it as
well. So let $f$ be some cut-continuous mapping from~$G$ to~$H$, let
$X_1$, \dots, $X_n$ be an $n/k$-cover and consider $X'_i$---a preimage of the
cut $X_i$ under~$f$. By definition, $X'_i$ is also a cut. If $e$~is an edge
of~$G$, $f(e)$ is an edge of~$H$, hence it is covered by at least~$k$ of the
cuts~$X_i$. Thus $e$ is covered by at least~$k$ of the cuts~$X'_i$. 
For the homomorphism part, one may observe that the mapping 
induced on edges by a homomorphism is cut-continuous \cite{DNR}, 
or just use the alternative definition in Equation~\eqref{eq:chiqalt}.
\end {proof}

As each graph $Q_{n/k}$ is a Cayley graph on~$\zet_2^n$, 
it follows~\cite{rs-thesis} that for every graph~$G$
the existence of a homomorphism from $G$ to $Q_{n/k}$ is 
equivalent to the existence of a cut-continuous mapping from $G$ to $Q_{n/k}$.
Consequently, we may as well use cut-continuous mapping
to~$Q_{n/k}$ in Equation~(\ref{eq:chiqalt}). This also provides an indirect
proof of Lemma~\ref{xTT}.

It is a standard exercise to show that $x(G)$ is 
the solution of the following linear
program ($\cal C$~denotes the family of all cuts in~$G$)
\begin{equation} \label{eq:LP1}
  \hbox {minimize $\displaystyle\sum_{X \in \cal C} w(X)$ subject to:
  for every edge $e$, 
  $\displaystyle\sum_{X, e \in X \in {\cal C}} w(X) \ge 1$.} 
\end{equation}

We conclude that we can replace $\inf$ by $\min$ in
the definition of $x(G)$---the infimum is always attained.
We can also consider the dual program
\begin{equation} \label{eq:LP2}
  \hbox {maximize $\displaystyle\sum_{e \in E(G)} y(e)$ subject to: 
  for every cut $X$, $\displaystyle\sum_{e, e \in X} y(e) \le 1$.} 
\end{equation}
This program is useful for computation of $x(G)$ for some~$G$. 
(Färnqvist, Jonsson, and Thapper~\cite{Svedi1} 
used a modification of this program.
There is an optimal solution $y^*$ of the above program, that respects
symmetries of $G$: if there is an automorphism of $G$ that maps 
edge $e$ to edge $f$, then $y^*(e) = y^*(f)$. This decreases the
size of the linear program for graphs with nontrivial automorphism group.)
Moreover, in the final section we use this dual program to discuss
yet another definition of $x(G)$ in terms of the bipartite subgraph
polytope.

There is another possibility to dualize the notion of fractional
cut covering, namely \emph{fractional cycle covering}. 
Bermond, Jackson and Jaeger~\cite{BJJ}
proved that every bridgeless graph has
a cycle $7/4$-cover (i.e., a collection of 7 cycles, that 
cover every edge at least 4 times), and Fan~\cite{Fan-mincc} 
proved that it has a $10/6$-cover. 
An equivalent formulation of the Berge-Fulkerson conjecture
claims that every cubic bridgeless graph has a
$6/4$-cover. On the other hand, Edmonds' characterization~\cite{Edmonds} of the matching
polytope implies that every cubic bridgeless graph has a cycle $3k/2k$-cover
(for some~$k$).
It is open, whether for some fixed $k$ every cubic bridgeless graph 
has a cycle $3k/2k$-cover. 

\section{Basic properties} \label{sec:basic} 

In this section we discuss how the cubical chromatic number relates to 
other graph parameters and prove analogs of some basic results about chromatic number. 

We let $\MC(G)$ denote the number of edges in the
largest cut in~$G$ and write 
$$
  b(G) = \MC(G)/|E(G)|
$$
for the \emph{bipartite density} of~$G$.
 
\begin {lemma}   \label{xb}
For any graph~$G$ it holds $x(G) \ge 1/b(G)$.
If $G$ is edge-transitive, then equality holds.
\end {lemma}

\begin {proof}
Suppose $x(G) = n/k$ and let $X_1$, \dots, $X_n$ be an $n/k$-cover.
Then $\sum_{i=1}^n |X_i| \le n\cdot b(G) |E(G)|$, on the other 
hand this sum is at least $k\cdot |E(G)|$, as every edge is counted
at least $k$~times. This proves the first part of the lemma. 
To prove the second part, let ${\cal X} = \{X_1, \dots, X_n\}$ be all cuts of
the maximal size (i.e., $|X_i| = b(G) |E(G)|$). From the edge-transitivity
follows that every edge is covered by the same number (say~$k$) of elements
of~$\cal X$. Now $k\cdot |E(G)| = \sum_{i=1}^n |X_i| = n \cdot b(G) |E(G)|$, 
which finishes the proof.
\end {proof}

\begin {corollary}   \label{xvalues}
Let $\Pet$ denote the Petersen graph. 
\begin{align*}
&x(K_{2n}) = x(K_{2n-1}) = 2 - 1/n & &\chiq(K_{2n}) = \chiq(K_{2n-1}) = 2n \\
&x(C_{2k+1}) = 1 + 1/(2k)          & &\chiq(C_{2k+1}) = 2 + 2/(2k-1) \\
&x(\Pet) = 5/4                     & &      \chiq(\Pet) = 8/3 
\end{align*}
\end {corollary}

In the following result, $g_o(G)$ denotes the \emph{odd girth}, that is, the length of 
a shortest odd cycle in~$G$. 

\begin {corollary}   \label{xbounds}
For any graph~$G$, 
$$ 
2 + \frac 2{g_o(G) -2} \le \chiq(G) \le 2\ceil{\frac {\chi(G)}2} \,.
$$
Equivalently, 
$1 + \frac {1}{g_o(G)-1} \le   x(G)  \le 2 - \frac 1{\lceil \chi(G)/2 \rceil}$.

In particular, 
$x(G) \in [1,2)$ and $\chiq(G) \ge 2$. 
\end {corollary}

\begin {proof}
Let $l=g_o(G)$, i.e., $C_l$ is the shortest odd cycle that is a subgraph
of~$G$. Put $n = \chi(G)$. Then there are homomorphisms 
$C_l \to G \to K_n$, so it remains to use 
Lemma~\ref{xTT} and Corollary~\ref{xvalues}.
\end {proof}

By combining Lemma~\ref{xTT} and Corollary~\ref{xvalues}
we get that there is no cut-continuous mapping from $K_{n+2}$ 
to $K_n$. As there is obviously a cut-continuous mapping 
(indeed, even a homomorphism) in the other direction, 
we conclude that the even cliques $K_{2n}$ form a strictly
ascending chain in the poset defined by cut-continuous mappings. 
This application was the original point in defining $x(G)$. The result
is not as straightforward as it appears (for example, there \emph{is}
a cut-continuous mapping $K_4 \to K_3$).

Next, we will study how good are the bounds of Corollary~\ref{xbounds}.
While they obviously are tight for $G$ equal to a complete graph, 
resp.\ odd cycle, they can be arbitrarily far off, as documented
by Corollary~\ref{smallx} and Theorem~\ref{xsparse}.
Before we get to that we need to look at $\chi_f(G)$---the fractional
chromatic number of~$G$. This may be defined by
$
  \chi_f(G) = \inf \{ n/k \mid  G \hom K(n,k) \} \,,
$
where $K(n,k)$ is the Kneser graph. 

\begin {lemma}   \label{Kneser}
Let $k$, $n$ be integers such that $0 < 2k \le n$. Then 
\begin{enumerate}
\item $b(K(n,k)) \ge 2k/n$.
\item $x(K(n,k)) \le n/(2k)$.
\end{enumerate}
Consequently, for any graph~$G$ we have
$x(G)    \le    \frac 12  \chi_f(G)$.
\end {lemma}

(Note that the bound is only useful if $k > n/4$.
Also note that the bound in part~1.\ is not optimal in general; the 
exact value of $b(K(n,k))$ is open~\cite{BollobasLeader}.) 

\begin {proof}
For the first part we let $U = \{ S \subseteq [n] \mid  1 \in S \}$ 
and observe that $\delta(U)$ contains $\binom {n-1}{k-1} \binom {n-k}{k}$~edges.
As Kneser graphs are edge-transitive, the second part follows by Lemma~\ref{xb}.
The rest follows by Lemma~\ref{xTT} and the definition of fractional
chromatic number.
\end {proof}

\begin {corollary}   \label{smallx}
For every $\eps > 0$ and every integer~$b$ there is a graph~$G$
such that 
$$
\chiq(G) < 2+\eps \quad 
    \hbox{and} \quad   \chi(G) > b \,.
$$
\end {corollary}

\begin {proof}
Let $G = K(n,k)$, for $n=2k+t$, $k=t^2$ and~$t$~large enough. 
Then by Corollary~\ref{Kneser} we have $x(G) \le n/2k = 1 + t/(2t^2)$, 
thus (for $t$ large enough) $\chiq(G) \le 2+\eps$. 
On the other hand, it is known \cite{Lovasz-Kneser} that $\chi(G) = n-2k+2 = t + 2$. 
Cf.\ also Corollary~\ref{thm:chif} below, where a stronger result is proved using 
semidefinite approximation. 
\end {proof}

By Corollary~\ref{xbounds}, we can view Corollary~\ref{smallx}
as a strengthening of the well-known fact~\cite{Erdos-highgirth} that there
are graphs with no short odd cycle and with a large chromatic number.
It also shows that the converse of Lemma~\ref{xTT} is far from being true: 
just take $G$~from the Corollary~\ref{smallx} and let $H = K_{b/2}$ (for $b$ large). 
Then $\chiq(G)$ is close to 2 and $\chiq(H)$ is at least $b/2$, still by an application of
Proposition~6.7 of~\cite{DNR} there is no cut-continuous mapping from $G$ to $H$.

It is interesting to find how various graph properties affect $\chiq(G)$.
From the values in Corollary~\ref{xvalues} we might think that $\chiq(G)$ is 
always larger than the fractional chromatic number~$\chi_f(G)$. However, 
this is very far from the truth, as shown in Corollary~\ref{thm:chif} below. 
We saw already that small $\chi(G)$ makes $\chiq(G)$ small
(Corollary~\ref{xbounds}), while
large $\chi(G)$ does not force $\chiq$ to be large (Corollary~\ref{smallx}). 
Also small $g_o(G)$ makes $\chiq(G)$ large (Corollary~\ref{xbounds} again).  
Complementing Corollary~\ref{smallx} we show 
that large $g_o(G)$ does not make $\chiq(G)$ small 
(but cf. Question~\ref{xsparsecubic}). 

\begin {theorem}   \label{xsparse}
For any integers $k$, $l$ there is a graph~$G$ such that
$\chiq(G) > k$ and $G$~contains no circuit of length at most~$l$.
\end {theorem}

\begin {proof}
We modify the famous Erd\H{o}s' proof of existence of high-girth
graphs of high chromatic number. 

Let $p=n^{\alpha -1}$ (where $\alpha \in (0,1/l)$) and consider the random graph $G(n,p)$.
The expected number of circuits of length at most~$l$ is $O((pn)^l) = o(n)$, 
therefore by Markov inequality with probability $1-o(1)$ 
the graph $G(n,p)$~contains at most $n$ circuits of length at most $l$. 

Using Lemma~\ref{brandom}, and in particular its Claim~1, where 
we put $\delta=n^{-\alpha/3}$ we get that a.a.s. 
$b(G(n,p)) \le \tfrac12 (1+O(n^{-\alpha/3}))$ and 
$|E(G(n,p))| > n^{1+\alpha}/3$. 

We take a graph~$G'$ satisfying all these three requirements. Then we 
delete one edge from each of the at most $n$ short circuits 
and let $G$ be the resulting graph. 


Clearly $G$ contains no short cycles. To show $\chi_q(G)$ is large it is enough 
to show that $x(G)$ can be arbitrary close to 2, or (using Lemma~\ref{xb}) 
to show that $b(G)$ can be arbitrary close to 1/2. 

As $|E(G')| = \Omega( n^{1+\alpha})$, and as we delete at most~$n$ edges 
of~$G'$ to get~$G$, we have $|E(G)| \ge |E(G')| (1-o(1))$. 
Obviously, MAXCUT in~$G$ cannot be larger than in $G'$, thus 
$b(G) \le b(G') (1+o(1)) = \tfrac 12 (1+o(1))$, which finishes the proof. 
\end {proof}

In the previous result it was crucial that the graphs had large degrees. 
For graphs of small degree the situation differs: 

\begin {question}   \label{xsparsecubic}
Let $G$ be a cubic graph with no cycle of length $\le c$. 
How large can $\chiq(G)$ (resp. $x(G)$) be?
\end {question}

For $c=3$, it follows from Brooks' theorem that $x(G) \le x(K_3) = 3/2$
($\chiq(G) \le 4$). 
For $c=17$, it is known~\cite{wpp} that $G$ has a cut-continuous mapping 
to $C_5$, hence $x(G) \le x(C_5) = 5/4$ ($\chiq(G) \le 8/3$).
Kardo\v{s}, Kr\'al' and Volec~\cite{KKV} prove that if the girth of a cubic graph~$G$ is large enough, then 
$x(G) \le 1.127752$ ($\chiq(G) \le 2.2929258651$). 
On the other hand, there is $\varepsilon > 0$ such that cubic graphs $G$ of arbitrary
high girth exist with $b(G) < 1-\eps$ (an unpublished result
of McKay, see also~\cite{rs-thesis}), hence with $x(G) > 1+\eps$ and so $\chiq(G) > 2+2\eps$.

\def\cprod{\mathop{\Box}} 
We conclude this section by a simple lemma that shows that $\chiq$ and $x$ enjoy some of the properties
of other chromatic numbers. (Here $G_1 \cprod G_2$ denotes the Cartesian product of graphs, 
$G_1 \times G_2$ the categorical one (also called tensor product); for more information 
about graph products we refer to Imrich and Klav\v{z}ar~\cite{products}.) 

\begin {lemma}   \label{operations}
\begin {enumerate}
\item $x(G) = \max \{ x(G')  \mid  
    \hbox {$G'$ is a component of~$G$} \}$
\item $x(G) = \max \{ x(G')  \mid  
    \hbox {$G'$ is a 2-connected block of~$G$} \}$ for a connected graph~$G$.
\item $x(G_1 \cprod G_2) = \max \{ x(G_1), x(G_2) \}$
\item $x(G_1 \times{}G_2) \le \min \{ x(G_1), x(G_2) \}$
\end {enumerate}
The same formulas are true for $\chiq$ in place of $x$.
\end {lemma}

\begin {proof}
We will prove that if $G'$, $G''$ are graphs that share at most one vertex, 
then $x(G' \cup G'') = \max \{ x (G'), x (G'') \}$. 
Clearly, this proves~1 and~2.
Let $x(G')=n/k$, and $x(G'')=m/l$ (by discussion after Equation~\eqref{eq:LP1} 
the infimum is attained) and suppose $X'_1$, \dots, $X'_n$ is an $n/k$-cover of~$G'$, 
while $X''_1$, \dots, $X''_m$ is an $m/l$-cover of $G''$. 
Consider the collection of $mn$ cuts 
$\{ X'_i \cup X''_j \}$ (these are cuts, indeed, as $G'$ and $G''$ share at
most one vertex). An edge of $G'$ is covered at least $mk$ times, an edge
of~$G''$ at least $nl$ times. Hence 
$x(G) \le \frac {mn}{\min \{ mk, nl \} } = 
  \max \{ \frac nk, \frac ml \} = \max \{ x(G'), x(G'')\}$. 
On the other hand, both $G'$ and $G''$ are subgraphs of~$G$, hence by
Lemma~\ref{xTT} the other inequality follows.

Part~3 follows from Lemma~\ref{xTT}, as between $G_1 \cprod G_2$ and $G_1 \cup G_2$
exists a cut-continuous mapping in both directions.

Part~4 follows from Lemma~\ref{xTT} as there are homomorphisms 
(and therefore $TT$~mappings) $G_1 \times{}G_2 \to G_i$ (for $i=1,2$). 

As $\chiq = 2/(2-x)$ (which is an increasing function for the values that
$x$ can attain), the results for $\chiq$ follow immediately.
\end {proof}

\section{Cubical chromatic number of random graphs}

In this section we consider the value of cubical chromatic number of random graphs. 
After a short technical lemma (that is also used in the proof of Theorem~\ref{xsparse}) 
we bound $\chi_q$ of a random graph $G(n,1/2)$ 
using a simple self-contained proof. We complement this by a result that provides 
the correct order of magnitude using results from Section~\ref{sec:SDP}. 

\begin {lemma}   \label{brandom}
Let $p$, $\delta$ be functions of~$n$ such that
$p, \delta \in [0,1]$ and $\delta^2 p \ge 7\log n/n$.
Then 
$b(G(n,p)) \le \frac 12 \left(1 + O(1/n) + O(\delta)\right)$ a.a.s.
In particular, we have 
$$
  b(G(n,p)) \le \frac 12 + O\left(\sqrt{\frac{\log n}{pn}}\right)
    \qquad \hbox{a.a.s.} 
$$
\end {lemma}

\begin {proof}
We will prove that almost all graphs have ``many edges but
no huge cut''.

\paragraph{Claim 1.} $|E(G(n,p))| > (1-\delta) p \binom{n}{2}$ a.a.s.

To prove this we use Chernoff inequality (as stated in Corollary~2.3
of~\cite{JLR-book}) for random variable $X = |E(G(n,p))|$. It claims
$\Pr[ |X - \Exp x| \ge \delta \Exp X] \le 2 e^{-\frac {\delta^2}{3} \Exp X}$ 
for $\delta \le 3/2$ and as $\Exp X = p \binom{n}{2}$, Claim~1 follows.

\paragraph{Claim 2.} $\MC(G(n,p)) < (1+\delta) p \frac {n^2}{4}$ a.a.s.

For a set $A \subseteq V(G(n,p))$ we let $X_A$~be the random variable
that counts the edges leaving~$A$, and put $a = |A| \le n/2$.
By Chernoff inequality for $X_A$ we easily get 
$$
  \Pr[X_A \ge (1+\delta) pn^2/4]
    \le 2 e^{-\frac {\delta^2}{3} p a (n-a)} 
    \le 2 e^{-\frac {\delta^2 p a n}{6}} \,. 
$$
It remains to estimate the total probability of a large cut:
$$
  \Pr[(\exists A) X_A \ge (1+\delta) pn^2/4]
     \le \sum_{a=1}^{n/2} \binom {n}{a} 2 e^{-\frac {\delta^2 p a n}{6}}  
     \le 2 \bigl((1+e^{-\frac {\delta^2 p n}{6}})^n - 1\bigr) \,.
$$
For $\delta^2 p \ge 7\log n/n$ the last expression tends to zero, 
which finishes the proof of Claim~2.
The rest of the proof of the lemma is a simple calculation.
\end {proof}


\begin {theorem}   \label{xrandom}
$$
 \Omega\left( \sqrt{n/\log n}\right) 
   \le \chiq(G(n,1/2)) 
   \le O\left({n}/{\log n} \right) \qquad \hbox{a.a.s.}
$$
\end {theorem}

\begin {proof}
The lower bound follows by Lemma~\ref{brandom},
the upper one by an application of Corollary~\ref{xbounds} and the
well-known fact that $\chi(G(n,1/2)) = O(n/\log n)$.
\end {proof}

The above theorem is included because the proof is short and self-contained. 
In the next theorem we give asymptotically tight estimate of~$\chi_q(G(n,p))$. 
In that, however, we rely on known estimates of $\vartheta(G(n,p))$ 
and the relation between $\vartheta$ and $\chi_q$ that we derive 
in Section~\ref{sec:SDP} below. 

\begin{theorem} \label{tightxrandom}
$\chiq(G(n,p)) = \Theta(\sqrt{pn})$ \qquad a.a.s.
\end{theorem}

\begin {proof}
The result follows directly using 
Theorem~\ref{t:chivrandom} 
and Theorem~\ref{t:chivchiq} below. 
\end {proof}

\section{Measuring the scale}   \label{sec:chiqQ}

In this section we will discuss the `invariance property' of cubical
chromatic number. In analogy with $\chi (K_n) = n$, $\chi_c(C_n^{{}^\ge k}) = n/k$, 
$\chi_f(K(n,k)) = n/k$, 
and `dimension of product of $n$~complete graphs is~$n$' we would
like to prove that $x(Q_{n/k}) = n/k$. The following lemma shows that the
situation is not so simple for $x$.

\begin {lemma}   \label{chiqinvareasy}
Let $1 \le k \le n$ be integers. Then we have $x(Q_{n/k}) \le \frac nk$.
If $k$ is odd, then $x(Q_{n/k}) \le \frac {n+1}{k+1}$.
\end {lemma}

\begin {proof}
For the first part, it suffices to consider the identical homomorphism
$Q_{n/k} \hom Q_{n/k}$. For the second part, mapping
$V(Q_{n/k}) \to  V(Q_{\frac{n+1}{k+1}})$ given by
$(x_1, \dots, x_n) \mapsto (x_1, \dots, x_n, x_1 + \cdots + x_n \bmod 2)$
is a homomorphism whenever $k$ is odd.
\end {proof}

Another complication is that by Corollary~\ref{xbounds} we have $x(G) < 2$
for any graph~$G$. However, with this exception, the bounds in 
Lemma~\ref{chiqinvareasy} are optimal: 

\begin {theorem}   \label{thm:chiqinvar}
Let $k$, $n$ be integers such that $k \le n \le 2k$.
Then 
\begin{enumerate}
\item if $k$ is even and $n < 2k$ then $x(Q_{n/k}) = \frac nk$; and 
\item if $k$ is odd then $x(Q_{n/k}) = \frac {n+1}{k+1}$. 
\end{enumerate}
\end {theorem}

\begin{corollary}
There is no homomorphism $Q_{n/k} \to Q_{n'/k'}$ 
if $1 \le n'/k' < n/k \le 2$, and $k$ is even. 
There is no homomorphism $Q_{tn/tk} \to Q_{n,k}$ 
if $t > 1$ is an integer, $tk$ is odd and $1 < n/k < 2$. 
\end{corollary}

\begin{proof}
The first part follows directly from Theorem~\ref{thm:chiqinvar} and Lemma~\ref{xTT}
(note that $(n'+1)/(k'+1) \le n'/k'$, so we do not care about parity of~$k'$). 
For the second part, observe first, that $Q_{n/k}$ is a subgraph of~$Q_{tn/tk}$ 
for a positive integer~$t$. It is known~\cite{UVC} that $Q_{tn/tk}$ is a core, 
thus it does not have a homomorphism to its proper subgraph. 
\end{proof}

This theorem was announced as a conjecture in the author's thesis~\cite{rs-thesis}, 
together with a part of a possible proof. The proof was finished by 
Engström, Färnqvist, Jonsson, and Thapper~\cite[Proposition~5.11]{Svedi2}, who did prove the 
inequality in Lemma~\ref{l:binomineq}.

We'll use the following result (see Lemma~13.7.4 and~13.1.2 of~\cite{AGT}).

\def\lambdamin{\lambda_{\rm min}} 
\begin {lemma}   \label{bipspectral}
Let $G$ be an $r$-regular graph with $n$ vertices, let $\lambdamin$ 
be the smallest eigenvalue of $G$. 
Then $b(G) \le \frac 12 (1-\frac \lambdamin r)$. 
\end {lemma}

The following lemma was proved (using a clever induction) 
by Engström, Färnqvist, Jonsson, and Thapper~\cite[Proposition~5.11]{Svedi2}, 
resolving thus a question from the author's thesis~\cite{rs-thesis}. 

\begin {lemma}   \label{l:binomineq}
Let $k$, $n$ be integers such that $k \le n < 2k$ and
$k$~is even, let $x$~be an integer such that $1 \le x \le n$.
Then 
$$
  \sum_{\mbox{odd $t$}} \binom xt \binom {n-x}{k-t} \le \binom{n-1}{k-1} \,.
$$
\end {lemma}

\begin {proof}(of Theorem~\ref{thm:chiqinvar})
Since Lemma~\ref{chiqinvareasy} provides the upper bound, we only need to establish the
lower bound. Suppose first that $k$~is even.
We shall use a spanning subgraph of $Q_{n/k}=Q_n^{\ge k}$, that contains
only edges of length precisely~$k$; we shall use $Q_n^{=k}$ to denote 
this subgraph. 

By Lemma~\ref{xTT} and~\ref{xb} we have that
$x(Q_{n/k}) \ge x(Q_n^{=k}) = 1/b(Q_n^{=k})$. By Lemma~\ref{bipspectral} 
it is enough to determine the smallest eigenvalue~$\lambdamin$ of~$Q_n^{=k}$. 
As $Q_n^{=k}$ is $\binom nk$-regular, we have 
$$
 \frac {1}{b(Q_n^{=k})} \ge 
   \frac {2}{1 - \lambdamin/\binom nk  } \,.
$$

It is standard (see, e.g., Problem 11.8 in~\cite{Lovasz-CPE}
or the theory of Association Schemes in Chapter 30 of \cite{LintWilson})  
that the eigenvalues of~$Q_n^{=k}$ are 
$$
   \sum_{t=0}^k (-1)^t \binom xt \binom {n-x}{k-t}  , 
$$
By using Vandermonde's identity and Lemma~\ref{l:binomineq}, we get that the above sum 
is at least $\binom nk (1-2k/n)$, which is equal to the sum for $x=1$.
Thus the smallest eigenvalue $\lambdamin$ equals $\binom nk (1-2k/n)$, 
and we obtain $x(Q_{n/k}) \ge n/k$ as desired.

For odd values of~$k$ we cannot use the same method, as 
then $Q_n^{=k}$~is bipartite, hence $b(Q_n^{=k})=1$. However, observe
that $Q_{\frac{n+1}{k+1}} \hom Q_{n/k}$, hence by Lemma~\ref{xTT} 
and the result for (even)~$k+1$ we have 
$$
  x(Q_{n/k}) \ge x(Q_{\frac{n+1}{k+1}}) \ge \frac{n+1}{k+1}\,.
$$
\end {proof}

\begin{corollary} 
The set $\{ x(G) \mid \hbox{$G$ is a graph}\}$ equals $\Q \cap [1,2]$. 
Consequently, the set 
$\{ \chiq(G)  \mid \hbox{$G$ is a graph}\}$ equals $\Q \cap [2, \infty)$. 
\end{corollary} 

\section{Semidefinite approximation} \label{sec:SDP} 

In this section we show how to approximate $\chi_q$ in polynomial time  
up to a factor of~$\pi/2$.  
Key to this approximation is the \emph{vector coloring}, introduced 
by \cite{KMS} based on the Lovász' $\vartheta$ function. 
The concept of vertex coloring is extended by using high-dimensional unit vectors as colors, 
and requiring adjacent vertices to be assigned distant vectors. 
Precisely: given a graph $G$ and real $t<0$ 
consider a mapping $f: V(G) \to \R^n$ (where $n=|V(G)|$), so that 
\begin{itemize}
  \item $\|f(v)\|_2 = 1$ for every vertex $v$ and 
  \item $\langle f(u), f(v) \rangle \le t$ for every edge $uv$. 
\end{itemize}
We let $t(G)$ denote the minimum $t$ such that function $f$ with 
the above properties exists. 
The \emph{vector chromatic number} of $G$ is defined as 
$\chi_v(G) = 1 - \frac{1}{t(G)}$. 

As these conditions for $t(G)$ can be formulated as a semidefinite program, 
the minimum indeed exists; more importantly, $t(G)$ can be approximated 
with an absolute error~$\eps$ in time polynomial in $n$ and $\log\tfrac 1\eps$. 
Indeed, Karger, Motwani and Sudan \cite[Lemma~3.2]{KMS} prove that if a graph $G$ has 
$\chi_v(G)=k$ then it is possible to find a vector $(k+\eps)$-coloring in 
time polynomial in $n$ and $\log 1/\eps$ --- in particular, one finds approximation 
to~$\chi_v$ up to an absolute error $\eps$. 

It is easy to see that $\chi_v(G) \le \chi(G)$ -- given 
a proper $k$-coloring, we may map all vertices of one color to 
one vertex of a simplex with $k$ vertices. This will lead 
to $t = -\frac{1}{k-1}$, and so indeed $\chi_v(G) \le k$. 
However, the fraction $\chi(G)/\chi_v(G)$ can be arbitrarily large \cite{FLS}, 
in fact as large as $n/polylog(n)$ (where $n = |V(G)|$); 
this contrasts sharply with Theorem~\ref{t:chivchiq}.  

For further properties of $\chi_v$ see \cite{KMS} and \cite{oghlan}. 
In the latter the following is shown.

\begin{theorem}[\cite{oghlan}]  \label{t:chivrandom} 
$c_1 \sqrt{np} \le \chi_v(G_{n,p}) \le c_2 \sqrt{np}$
with probability $1-o(1)$. 
\end{theorem}

Now we proceed to show to connection between $\chi_q$ and $\chi_v$. 

\begin{theorem} \label{t:chivchiq} 
For every graph $G$ we have 
$$
    \chi_v(G) \le \chi_q(G) \le \frac{\pi}{2} \chi_v(G) \,.  
$$
\end{theorem}

\begin{proof} 
We prove the lower bound first. Recall that $\chi_q(G) = \frac{2}{2-x(G)}$ 
and $x(G) = n/k$, for some $n$, $k$ where there is a $k$-cover of~$G$ by $n$~cuts. 
(The fact that the infimum in the definition of $x(G)$ is attained follows 
from the linear-programming reformulation, see Equation~\eqref{eq:LP1}.) 
Equivalently, there is a mapping $g:V(G) \to \{ \pm 1\}^n$ (the $i$-th 
coordinate encodes the $i$-th cut so that for every edge $uv$ the
vectors $g(u)$ and $g(v)$ differ in $\ge k$ coordinates. 
Put $f(v) = g(v)/\sqrt{n}$. Obviously, each $f(v)$ is a unit vector, 
while for every edge $uv$ we have 
$$
  \langle f(u), f(v) \rangle = 1 - \frac{2d_H( g(u), g(v) )}{n} \le 1 - \frac{2k}{n} = 1 - \frac 2{x(G)} \,. 
$$
Therefore, for this $f$ we get $t \le 1-2/x(G)$. Consequently, 
$$ 
 \chi_v(G) \le  1 - \frac{1}{t} \le 1 - \frac{x(G)}{x(G)-2} = \frac2{2-x(G)} = \chi_q(G) \,. 
$$  

For the upper bound we use probabilistic approach, motivated by 
the algorithm for approximating MAXCUT by Goemans and Williamson \cite{GW}. 
Consider a mapping $f$ as above, the scalar products are at most $t$ with 
$\chi_v(G) = 1-1/t$. 
For a large $N$, we choose $N$ uniformly random hyperplanes in $\R^n$ through the origin. 
With probability $1$ none of them contains any of the points $f(v)$ for $v \in V(G)$, 
therefore each hyperplane defines a cut. We shell prove that with probability $1-o(1)$ 
this cut covering gives us the desired bound. 

To this end, consider an edge $uv \in E(G)$, let $\alpha$ be the angle between 
the unit vectors $f(u)$ and $f(v)$. The following elementary observation
(used also in~\cite{GW}) is crucial for the calculation: 
\begin{center}
A random hyperplane through origin separates $f(u)$ and $f(v)$ with 
probability $\frac{\alpha}{\pi}$. 
\end{center}
For an edge $e=uv$ let $X_e$ be the random variable that counts how many of the $N$ hyperplanes
separate the end-vertices of $e$. Obviously, $X_e$ follows a binomial distribution $Bin(N,p)$ 
with $p = \frac{\alpha}{\pi}$. We have $\cos \alpha = \langle f(u), f(v) \rangle \le t$, so 
$p \ge \frac{\arccos t}{\pi}$. 
By the Chernoff inequality we have $Pr[ X_e < pN - s ] < e^{-\frac{s^2}{2Np}}$. 
Putting $s = \lceil N^{2/3}\rceil$ 
we obtain 
$$
  Pr[ X_e < pN - \lceil N^{2/3} \rceil ] < e^{-\frac{N^{1/3}}{2p}} = o(1) 
$$
(the $o(1)$ is with respect to $N$ growing to infinity). 
Thus, with probability $1 - \binom n2 o(1) = 1-o(1)$ we have 
$X_e \ge pN - \lceil N^{2/3} \rceil$ for every edge $e$. So for every large enough $N$ there 
is a cut covering achieving this and from the definition of $x(G)$, we get that 
$$
  x(G) \le \frac{N}{pN - \lceil N^{2/3} \rceil} = \frac {1}{p} (1+o(1)) \,. 
$$ 
As we may choose arbitrarily large $N$, we get from here that $x(G) \le \frac 1p = \frac \pi{\arccos t}$. 
Now from the definition we obtain 
$$
  \frac{\chi_q(G)}{\chi_v(G)} = \frac{\frac {2}{2-x(G)}}{1- \frac 1t} 
    \le \frac{\frac {2}{2- \frac {\pi}{\arccos t} }}{1-\frac 1t}   
    = \frac{t \arccos t}{(\arccos t - \pi/2)(t-1)}
$$
Putting $t=\cos \alpha$ and $\beta = \alpha - \pi/2$ (so that $t = -\sin \beta$), 
the last expression equals
$$
  \frac{\sin \beta}{\beta} \frac{\beta + \frac \pi2}{\sin \beta + 1} \le  1\cdot \frac{\pi}{2} 
$$
(we used the elementary estimate $\tfrac 2\pi \beta \le \sin \beta \le \beta$
valid for $\beta \in [0, \pi/2]$). 
\end{proof} 

We note that the above proof also yields bound 
$\chi_q(G) \le {1}/\bigl(1- \frac{\pi}{2\arccos \frac1{1-\chi_v(G)}}\bigr)$, which 
is, for small values of $\chi_v(G)$, slightly better than the above theorem.

\begin{corollary} \label{thm:apx} 
There is a polynomial-time algorithm that approximates $\chi_q(G)$ with 
approximation factor almost $\frac\pi2$. More precisely: to get 
an approximation factor at most $\frac\pi2 (1+\eps)$ we need an 
algorithm polynomial in $|V(G)|$ and $\log 1/\eps$. 
\end{corollary}

\begin{corollary} \label{thm:chif} 
For every graph $G$ we have 
$$
    \chi_q(G) \le \frac{\pi}{2} \chi_f(G) \,.  
$$
Moreover, there is a sequence of graphs for which $\chi_q(G)$ is bounded, 
while $\chi_f(G)$ is unbounded. 
\end{corollary}

\begin{proof}
For the first part it is enough to use Theorem~\ref{t:chivchiq}, the 
bound $\chi_v(G) \le \vartheta(\overline G)$ (Theorem~8.2 of~\cite{KMS}) 
and the well-known bound $\vartheta(\overline G) \le \chi_f(G)$. 
We use Theorem~1.2 of Feige, Langberg, and Schechtman~\cite{FLS}: 
There are infinitely many graphs $G$ that are vector 3-colorable and satisfy
$\alpha(G) \le n^{0.843}$ (where $n$ is the number of vertices of~$G$). 
Each such graph~$G$ satisfies $\chi_q(G) \le 3\pi/2 < 5$, 
and $\chi_f(G) \ge n/n^{0.843} = n^{0.157}$. 
\end{proof}

Let us note here an exciting development related to the above mentioned result 
of Feige, Langberg, and Schechtman~\cite{FLS}. If we are given a 3-colorable graph~$G$, 
it is still computationally hard to find a 3-coloring of it. This lead 
Karger, Motwani, and Sudan~\cite{KMS} to their definition of vector chromatic number. 
As $\chi_v(G) \le \chi(G) \le 3$, and vector coloring is computationally tractable, 
we can find a vector 3-coloring of~$G$ and then use various rounding techniques to 
find a coloring of~$G$. The best result in this direction is $O(n^{0.19996})$ colors 
due to Kawarabayashi and Thorup~\cite{KT}. 
As a limit to this approach Feige, Langberg, and Schechtman~\cite{FLS} observe that 
just using the fact that $\chi_v(G) \le 3$ does not prevent a graph from having 
chromatic number as large as $\Omega(n^{0.157})$, thus to efficiently color 3-colorable 
graph with less colors (if at all possible), a different technique is needed.

\section{Concluding Remarks}

\paragraph{Bipartite subgraph polytope}   

For a bipartite subgraph $B \subseteq G$, let $c_B$ be the
characteristic vector of~$E(B)$. Bipartite subgraph polytope
$P_B(G)$ is the convex hull of points $c_B$, for all bipartite
graphs $B \subseteq G$.
The study of this 
polytope was motivated by the MAXCUT problem:
looking for a weighted maximum cut of~$G$ simply means solving a linear
program over~$P_B(G)$. Thus, for graphs where $P_B(G)$ has
simple description, we can have a polynomial-time algorithm for
MAXCUT; this in particular happens for weakly bipartite graphs
(which include planar graphs), see~\cite{GrPul}.
We apply $P_B$ to yield yet another definition of~$x$.

\begin {theorem}   \label{xfacets}
$
   x(G) = \max \{ \sum_{e\in E(G)} y_e \mid 
       \hbox{$y \cdot c \le 1$ defines a facet of~$P_B(G)$} \} 
$
\end {theorem}

\begin {proof}
By LP duality~$x(G)$ is a solution to the program~(\ref{eq:LP2}).
This means, that we are maximizing over such $y$, that
for each cut~$X$ satisfy $y \cdot c_X \le 1$.
As the convex hull of vectors~$c_X$ is~$P_B$, we are maximizing
the sum of coordinates of an element of the dual polytope~$P_B^*$. This maximum
is attained for some vertex of~$P_B^*$, that is for $y$ 
such that $y \cdot c \le 1$ defines a facet of~$P_B$.
\end {proof}

`Natural' facets of~$P_B(G)$ are defined by 
$\sum_{e \in E(H)} y_e \le \MC(H)$ for some $H \subseteq G$. 
(This inequality is satisfied for every graph $H$, but it
doesn't always define a face of maximal dimension.)
This proves the following observation (we add a direct proof, too).

\begin {lemma}   \label{xB}
$
  x(G) \ge  1/(\min_{H \subseteq G} b(H)) 
$
\end {lemma}

\begin {proof}
Suppose $H \subseteq G$. Then there exists a cut-continuous mapping (indeed, a homomorphism) 
from~$H$ to~$G$, which by Lemma~\ref{xTT} and~\ref{xb} implies $1/b(H) \le x(G)$.
\end {proof}

Let us return to Lemma~\ref{xb} for a while. In general
$x(G)$ and $1/b(G)$ can be as distant as possible: Let $G$ be a
disjoint union of a $K_n$ and $K_{N,N}$. Now $x(G)$ is close to~2
(because $G$~is homomorphically equivalent to~$K_n$, hence $x(G) =
x(K_n)$) and $b(G)$ is close to~1 (provided $N$ is sufficiently large).
This motivates Lemma~\ref{xB}, which improves the original bound.
A natural question is whether this improvement gives the correct
size of~$x$.
It turns out it does not (contrary to a conjecture in the 
author's thesis). In \cite{Svedi2} it is shown, that the
circular clique $K_{11/4}$ is a counterexample.

\paragraph{A failed approach} 
The proof of Theorem~\ref{thm:chiqinvar} could be attempted by another way:
First, observe that the Kneser graph $K(n,r)$ is a subgraph of~$Q_{n/2r}$.
By Lemmas~\ref{xTT} and~\ref{xb} we have $x(Q_{n/2r}) \ge x(K(n,r)) \ge \frac 1{b(K(n,r))}$.
Thus, if we knew the value of $b(K(n,r))$ (and it turned out to be $2r/n$ for 
the range of $r$ we are interested in), we would be done.

In~\cite{PT} it is claimed that if $2r \le n \le 3r$ then,
indeed, $b(K(n,r)) = 2r/n$.
This would imply the conjecture for even~$k$ less than $3/2 \cdot n$;
unfortunately the proof in~\cite{PT} is incomplete (as already observed
by~\cite{BollobasLeader}). Thus, the true value of MAXCUT for Kneser graphs
remains open. 

\paragraph{Generalizations and future work}
As already mentioned in the introduction, the metric that is used 
in \cite{Svedi1, Svedi2} to study approximability of MAX-$H$-COLORING
can be computed from a generalization of fractional covering 
by cuts. One only needs to consider more general edge sets in place of cuts, 
namely edge sets of graphs that are homomorphic to $H$. Then the cube $Q_{n/k}$
in Equation~\eqref{eq:chiqalt} is replaced by appropriately defined
power of~$H$. One may also use this motivation to define $H$-continuous
mappings as follows. We call a subset $X \subseteq E(G)$ an $H$-cut
in $G$ whenever there is a mapping $g : V(G) \to V(H)$ for which
$g^{-1}(E(H)) = X$. We say a mapping $f: E(G_1) \to E(G_2)$ is
\emph{$H$-continuous} whenever a preimage of each $H$-cut is
an $H$-cut. This notion deserves further attention. 
Some preliminary observations are obtained in~\cite{Svedi2}. 

\paragraph{Possible use of recent techniques for approximating MAXCUT} 
In recent years, a lot of attention has been put to various ways to approximate 
MAXCUT without using semidefinite programming. 
In particular, finding a combinatorial approximation algorithm is of interest; 
a nice algorithm based on random walks exists~\cite{KS}. 
It seems, however, that this method fails to approximate~$\chiq$, because it is based on 
local properties of the given graph, and $\chiq$~can be very large 
even for graphs that are locally trees (see Theorem~\ref{xsparse}). 
It would be interesting, though, if such techniques could be used for approximating~$\chiq$ 
for graphs of bounded degree. 

\paragraph{Number of cuts required} 

By definition, if $x(G) = t$ then there is a cut $n/k$-cover 
for some $n$, $k$ satisfying $t = \tfrac nk$. It would be nice 
to know how large $n$ is required. To be precise, define 
$n(G)$ to be the smallest $n$ as above. Then we let 
$$
  f(v) = \max \{ n(G) \mid \hbox{$G$ is a graph with $v$ vertices} \}  \,. 
$$
This maximum clearly exists (as there are only finitely many 
graphs on $v$ vertices). 

\begin{question}
  How fast doest $f(v)$ grow? Is $f(v) \le 2^v$? For what graph(s) is the maximum in the definition 
  of~$f(v)$ attained? 
\end{question}

The estimate by $2^v$ seems natural, as there is only $2^{v-1}$ 
different cuts in a graph on $v$ vertices. However, one may be 
forced to take some cuts repeatedly. 

\paragraph{Complexity} In view of the complexity of computing other variants of chromatic number, 
the following conjecture is natural. Note, however, that in contrast with 
chromatic or fractional chromatic number, cubical chromatic number can be 
approximated up to a constant factor. 

\begin{conjecture}
For any $s>2$ determining if an input graph $G$ satisfies 
$\chi_q(G) \le s$ is NP-complete. 
\end{conjecture}

Perhaps more importantly: how well can one approximate $\chi_q$ in polynomial time? 
Can one use the techniques of~\cite{AKK} to find a PTAS for $\chi_q(G)$ --- at least in the case 
when $G$~is dense? It is tempting to use the ellipsoid method to solve 
the linear program~\eqref{eq:LP2}, where results of~\cite{AKK} can serve as an 
(approximate) separation oracle. To do this, however, we need a PTAS 
for weighted MAX-CUT. While some results in this direction are known~\cite{VK}, 
they are not strong enough (the issue is that some weights may be much larger than the others, 
which basically makes our instance not dense).

\paragraph{Cubic graphs} For the reader's convenience we restate here 
Question~\ref{xsparsecubic}. For known partial results we refer the reader 
to Section~\ref{sec:basic}. 

\begin {question}   
Let $G$ be a cubic graph with no cycle of length $\le c$. 
How large can $x(G)$ (resp. $\chiq(G)$) be?
\end {question}

\section*{Acknowledgement}

I thank Matt DeVos and Stephan Thomassé for suggesting to use 
semidefinite programming to approximate the cubical chromatic number 
and for helpful discussion about the topic. 
I thank the anonymous referees for helpful comments leading to improved presentation. 

\bibliographystyle{rs-amsplain}
\bibliography{chiq}

\providecommand{\bysame}{\leavevmode\hbox to3em{\hrulefill}\thinspace}
\providecommand{\MR}{\relax\ifhmode\unskip\space\fi MR }
\providecommand{\MRhref}[2]{%
  \href{http://www.ams.org/mathscinet-getitem?mr=#1}{#2}
}
\providecommand{\href}[2]{#2}
\begin{thebibliography}{10}

\bibitem{AKK}
Sanjeev Arora, David Karger, and Marek Karpinski, \emph{Polynomial time
  approximation schemes for dense instances of {NP}-hard problems}, Proceedings
  of the twenty-seventh annual ACM symposium on Theory of computing, ACM, 1995,
  pp.~284--293.

\bibitem{BJJ}
Jean-Claude Bermond, Bill Jackson, and Fran{\c{c}}ois Jaeger, \emph{Shortest
  coverings of graphs with cycles}, J. Combin. Theory Ser. B \textbf{35}
  (1983), no.~3, 297--308.

\bibitem{BollobasLeader}
B\'ela Bollob\'as and Imre Leader, \emph{Set systems with few disjoint pairs},
  Combinatorica \textbf{23} (2003), no.~4, 559--570.

\bibitem{oghlan}
{Amin} {Coja-Oghlan}, \emph{The {L}ovász {N}umber of {R}andom {G}raphs},
  Combinatorics, Probability and Computing \textbf{14} (2005), no.~04,
  439--465.

\bibitem{wpp}
Matt DeVos and Robert {\v{S}}{\'a}mal, \emph{High-girth cubic graphs are
  homomorphic to the {C}lebsch graph}, J.~Graph Theory \textbf{66} (2011),
  241--259.

\bibitem{DNR}
Matt DeVos, Jaroslav~Ne\v set\v ril, and Andr\'e Raspaud, \emph{On edge-maps
  whose inverse preserves flows and tensions}, Graph Theory in Paris:
  Proceedings of a Conference in Memory of Claude Berge (J.~A. Bondy,
  J.~Fonlupt, J.-L. Fouquet, J.-C. Fournier, and J.~L.~Ramirez Alfonsin, eds.),
  Trends in Mathematics, Birkh\"auser, 2006.

\bibitem{Diestel}
Reinhard Diestel, \emph{Graph theory}, Graduate Texts in Mathematics, vol. 173,
  Springer-Verlag, New York, 2000.

\bibitem{Edmonds}
Jack Edmonds, \emph{Maximum matching and a polyhedron with {$0,1$}-vertices},
  J. Res. Nat. Bur. Standards Sect. B \textbf{69B} (1965), 125--130.

\bibitem{Svedi2}
Robert Engström, Tommy Färnqvist, Peter Jonsson, and Johan Thapper, \emph{An
  approximability-related parameter on graphs---properties and applications},
  Discr. Math. Th. Comp. Sc. \textbf{17} (2015), no.~1, 33--66.

\bibitem{Erdos-highgirth}
Paul Erd{\H{o}}s, \emph{Graph theory and probability}, Canad. J. Math.
  \textbf{11} (1959), 34--38.

\bibitem{Fan-mincc}
Genghua Fan, \emph{Minimum cycle covers of graphs}, J. Graph Theory \textbf{25}
  (1997), no.~3, 229--242.

\bibitem{FLS}
Uriel Feige, Michael Langberg, and Gideon Schechtman, \emph{Graphs with tiny
  vector chromatic numbers and huge chromatic numbers}, SIAM J. Comput.
  \textbf{33} (2004), no.~6, 1338--1368.

\bibitem{VK}
Wenceslao Fernandez de~la Vega and Marek Karpinski, \emph{Polynomial time
  approximation of dense weighted instances of {MAX}-{CUT}}, Random Structures
  \& Algorithms \textbf{16} (2000), no.~4, 314--332.

\bibitem{Svedi1}
Tommy Färnqvist, Peter Jonsson, and Johan Thapper, \emph{Approximability
  distance in the space of {$H$}-colourability problems.}, Proceedings of the
  4th International Computer Science Symposium in Russia (CSR-2009), 2009,
  arXiv:0802.0423.

\bibitem{UVC}
Chris Godsil, David Roberson, Brendan Rooney, Robert \v{S}\'{a}mal, and
  Antonios Varvitsiotis, \emph{Rigidity and {G}raph {T}heory {III}:
  {C}ategorical {P}roducts and {U}nique {H}omomorphisms}, in preparation, 2015.

\bibitem{AGT}
Chris Godsil and Gordon Royle, \emph{Algebraic graph theory}, Graduate Texts in
  Mathematics, vol. 207, Springer-Verlag, New York, 2001.

\bibitem{GW}
Michel~X. Goemans and David~P. Williamson, \emph{Improved approximation
  algorithms for maximum cut and satisfiability problems using semidefinite
  programming}, J. Assoc. Comput. Mach. \textbf{42} (1995), no.~6, 1115--1145.

\bibitem{GrPul}
Martin Gr{\"o}tschel and William~R. Pulleyblank, \emph{Weakly bipartite graphs
  and the max-cut problem}, Oper. Res. Lett. \textbf{1} (1981/82), no.~1,
  23--27.

\bibitem{products}
Wilfried Imrich and Sandi Klav\v{z}ar, \emph{Product graphs, structure and
  recognition}, Wiley, 2000.

\bibitem{JLR-book}
Svante Janson, Tomasz {\L}uczak, and Andrzej Rucinski, \emph{Random graphs},
  Wiley-Interscience Series in Discrete Mathematics and Optimization,
  Wiley-Interscience, New York, 2000.

\bibitem{KS}
Satyen Kale and Comandur Seshadhri, \emph{Combinatorial {A}pproximation
  {A}lgorithms for {M}ax{C}ut using {R}andom {W}alks}, Innovations in Computer
  Science -- {ICS} 2010, Tsinghua University, Beijing, China, January 7--9,
  2011. Proceedings, 2011, pp.~367--388.

\bibitem{KKV}
Franti\v{s}ek Kardo\v{s}, Daniel Kr\'al', and Jan Volec, \emph{Maximum
  edge-cuts in cubic graphs with large girth and in random cubic graphs},
  Random Structures \& Algorithms \textbf{41} (2012), no.~4, 506--520.

\bibitem{KMS}
David Karger, Rajeev Motwani, and Madhu Sudan, \emph{Approximate graph coloring
  by semidefinite programming}, J. ACM \textbf{45} (1998), no.~2, 246--265.

\bibitem{KT}
Ken{-}ichi Kawarabayashi and Mikkel Thorup, \emph{Coloring 3-colorable graphs
  with $o(n^{1/5})$ colors}, 31st International Symposium on Theoretical
  Aspects of Computer Science {(STACS} 2014), {STACS} 2014, March 5-8, 2014,
  Lyon, France, 2014, pp.~458--469.

\bibitem{Lovasz-Kneser}
L{\'a}szl{\'o} Lov{\'a}sz, \emph{Kneser's conjecture, chromatic number, and
  homotopy}, J. Combin. Theory Ser. A \textbf{25} (1978), no.~3, 319--324.

\bibitem{Lovasz-CPE}
L{\'a}szl{\'o} Lov{\'a}sz, \emph{Combinatorial problems and exercises},
  North-Holland Publishing Co., Amsterdam, 1979.

\bibitem{NS-TT}
Jaroslav Ne{\v{s}}et{\v{r}}il and Robert {\v{S}}{\'a}mal, \emph{On
  tension-continuous mappings}, European J. Combin. \textbf{29} (2008), no.~4,
  1025--1054.

\bibitem{PT}
Svatopluk Poljak and Zsolt Tuza, \emph{Maximum bipartite subgraphs of {K}neser
  graphs}, Graphs Combin. \textbf{3} (1987), no.~2, 191--199.

\bibitem{chiq-ea}
Robert {\v S}\'amal, \emph{Fractional covering by cuts}, Electronic Notes in
  Discrete Mathematics (2005), no.~22, 455--459, Proceedings of 7th
  International Colloquium on Graph Theory (Hy{\`e}res, 2005).

\bibitem{rs-thesis}
Robert {\v S}\'amal, \emph{On {XY} mappings}, Ph.D. thesis, Charles University,
  2006.

\bibitem{LintWilson}
Jack~H. van Lint and Richard~M. Wilson, \emph{A course in combinatorics},
  second ed., Cambridge University Press, Cambridge, 2001.

\end{thebibliography}

\end{document}